\documentclass{amsart}

\usepackage{amsmath}
\usepackage{amssymb}
\usepackage{amsthm}
\usepackage[all,cmtip]{xy}

\theoremstyle{plain}
\newtheorem{thm}{Theorem}[section]
\newtheorem{prop}[thm]{Proposition}

\newtheorem{cor}[thm]{Corollary}

\theoremstyle{definition}
\newtheorem{defn}[thm]{Definition}
\newtheorem{eg}[thm]{Example}

\theoremstyle{remark}
\newtheorem{rem}[thm]{Remark}

\DeclareMathOperator{\Ext}{Ext}

\DeclareMathOperator{\Sing}{Sing}
\DeclareMathOperator{\Spec}{Spec}

\DeclareMathOperator{\Ker}{Ker}
\DeclareMathOperator{\Image}{Im}
\DeclareMathOperator{\Gal}{Gal}

\DeclareMathOperator{\Gr}{Gr}

\begin{document}

\title
[$\mathbb{Q}$-smoothings of threefolds]
{On deformations of $\mathbb{Q}$-Fano threefolds II}

\author{Taro Sano}
\address{Max Planck Institute for Mathematics, Vivatsgasse 7, 53111 Bonn, Germany}
\email{tarosano222@gmail.com}

\maketitle

\begin{abstract}
We investigate some coboundary map associated to a $3$-fold terminal singularity which is important in the study of 
deformations of singular $3$-folds. 
We prove that this map vanishes only for quotient singularities and 
a $A_{1,2}/4$-singularity, that is, a terminal singularity analytically isomorphic to 
a $\mathbb{Z}_4$-quotient of the singularity
$ (x^2+y^2 +z^3+u^2=0)$. 

As an application, we prove that a $\mathbb{Q}$-Fano $3$-fold with terminal singularities can be deformed to one 
with only quotient singularities and $A_{1,2}/4$-singularities. 
We also treat the $\mathbb{Q}$-smoothability problem on $\mathbb{Q}$-Calabi--Yau $3$-folds. 
\end{abstract}

\tableofcontents

\section{Introduction}
We consider algebraic varieties over the complex number field $\mathbb{C}$. 

This paper is a continuation of \cite{Sano}. 
We study the $\mathbb{Q}$-smoothability of a $\mathbb{Q}$-Fano $3$-fold $X$ via certain coboundary maps of local cohomology groups 
associated to the singularities on $X$. 

\subsection{$\mathbb{Q}$-smoothing of $\mathbb{Q}$-Fano $3$-folds}
In this paper, {\it a $\mathbb{Q}$-Fano $3$-fold} means a projective $3$-fold with only terminal singularities whose 
anticanonical divisor is ample. 
A $\mathbb{Q}$-Fano $3$-fold is an important object in the classification theory of algebraic $3$-folds. 
It is one of the end products of the Minimal Model Program. 
Toward the classification of $\mathbb{Q}$-Fano $3$-folds, it is fundamental to study their deformations. 

Locally, a $3$-fold terminal singularity has {\it a $\mathbb{Q}$-smoothing}, that is, 
it can be deformed to a variety with only quotient singularities. 
In general, local deformations of singularities may not lift to a global deformation of a projective $3$-fold  
as shown for Calabi--Yau $3$-folds (cf.\ \cite[Example 5.8]{Namtop}). 
Nevertheless, Alt{\i}nok--Brown--Reid (\cite[4.8.3]{ABR}) conjectured that a $\mathbb{Q}$-Fano $3$-fold has a $\mathbb{Q}$-smoothing. 
(See Example \ref{eg:A12hypersurface} for an example of a $\mathbb{Q}$-smoothing.)
This conjecture aims to reduce the classification of $\mathbb{Q}$-Fano $3$-folds to those with only quotient singularities. 
 For example, there are several papers (cf.\ \cite{TandJ}, \cite{TakagiDV}) 
 on the classification of certain $\mathbb{Q}$-Fano $3$-folds with only quotient singularities. 

 Previously, deformations of $\mathbb{Q}$-Fano $3$-folds are treated in several papers 
 (cf. \cite{NamFano}, \cite{mina}, \cite{takagi2}, \cite{Sano}). 
In \cite[Theorem 1.5]{Sano}, the author proved that a $\mathbb{Q}$-Fano $3$-fold with only ``ordinary'' terminal singularities has a $\mathbb{Q}$-smoothing.  
 (See Definition \ref{ordinarydefn} for the ordinariness of the singularity.) 
 In this article, we treat the remaining case, that is, a $\mathbb{Q}$-Fano $3$-fold with non-ordinary terminal singularities. 
We can deform the non-ordinary terminal singularities except one special singularity as follows. 

\begin{thm}\label{qsmqfanothmintro}
	A $\mathbb{Q}$-Fano $3$-fold can be deformed to one with only quotient singularities and 
	  $A_{1,2}/4$-singularities.  
	\end{thm}

Here, an {\it $A_{1,2}/4$-singularity} means 
a singularity analytically isomorphic to \[
0 \in (x^2+y^2+z^3+u^2=0)/\mathbb{Z}_4 \subset \mathbb{C}^4/\mathbb{Z}_4 (1,3,2,1),
\] 
where $x,y,z,u$ are coordinates on $\mathbb{C}^4$ and $\mathbb{C}^4/\mathbb{Z}_4 (1,3,2,1)$ is the quotient of $\mathbb{C}^4$ 
by an action of $\mathbb{Z}_4=\langle \sigma \rangle$ as follows: 
\[
\sigma \cdot (x,y,z,u) = (\sqrt{-1} x, -\sqrt{-1} y, -z, \sqrt{-1} u). 
\]
Although we do not know how to deal with $A_{1,2}/4$-singularities, we believe that Theorem \ref{qsmqfanothmintro} is useful for the classification.  

\begin{rem}
	The author studied a deformation of a $\mathbb{Q}$-Fano $3$-fold with its anticanonical element in \cite{Sano} and \cite{Sano-elephant}. 
	In \cite[Theorem 1.3]{Sano-elephant}, it is proved that, if a $\mathbb{Q}$-Fano $3$-fold $X$ has a member $D \in |{-}K_X|$ with only isolated singularities, 
	then $X$ has a $\mathbb{Q}$-smoothing. In the proof, it is necessary to use \cite[Theorem 1.9]{Sano} and Theorem \ref{qsmqfanothmintro} in this paper. 
	
	The existence of an elephant with mild singularities is discussed in \cite[Section 4]{Sano-elephant} by showing several examples of 
	$\mathbb{Q}$-Fano $3$-folds.   
	\end{rem}

\subsection{Methods of the proof}
We use a method which is used in \cite[Theorem 3.5]{Sano}. 
Let $(U,p)$ be a germ of a $3$-fold terminal singularity.  
The key tool of our method is the coboundary map $\phi_{U}$ associated to some local cohomology group on a birational modification $\tilde{U} \rightarrow U$.  
(See (\ref{phiUdescription}) for the definition of $\phi_{U}$.)  
If this map is nonzero, it is useful for finding a smoothing or a $\mathbb{Q}$-smoothing of a projective $3$-fold. 
(cf.\  \cite{NamSt}, \cite{mina}, \cite{Sano})  
The following purely local statement is the main result of Section \ref{coboundarysection}. 

\begin{thm}\label{coboundarytheoremintro}
	Let $(U,p)$ be a germ of a $3$-fold terminal singularity which is not a quotient singularity. 
	
	Then $\phi_U =0$ if and only if $(U,p)$ is an $A_{1,2}/4$-singularity. 
\end{thm} 
 
The map $\phi_U$ is known to be nonzero when $(U,p)$ is Gorenstein (\cite[Theorem 1.1]{NamSt}) or 
$(U,p)$ is an ordinary singularity (\cite{mina}, \cite{Sano}). We calculate the coboundary map for 
 a non-ordinary singularity. 
 
Let us mention about the proof of Theorem \ref{coboundarytheoremintro}. 
Since a terminal singularity $(U,p)$ of index $r$ is a $\mathbb{Z}_r$-quotient of a hypersurface singularity $(V,q)$, 
the set $T^1_{(U,p)}$ of first order deformations of $(U,p)$ is the $\mathbb{Z}_r$-invariant part of $T^1_{(V,q)}$. 
The set $T^1_{(V,q)}$ can be written as $\mathcal{O}_{V,q}/ J_{V,q}$ for the Jacobian ideal of $(V,q)$. 
We calculate the map $\phi_U$ by using this structure and the inequality (\ref{ineqkerim}) proved in \cite{NamSt}.

By Theorem \ref{coboundarytheoremintro} (ii), the map $\phi_U$ vanishes for a neighborhood $U$ of an $A_{1,2}/4$-singularity. 
It seems that we need a new method to treat a $\mathbb{Q}$-Fano $3$-fold with $A_{1,2}/4$-singularities. (See Remark \ref{A12remark})

\subsection{$\mathbb{Q}$-smoothing of $\mathbb{Q}$-Calabi--Yau $3$-folds}

As another corollary of Theorem \ref{coboundarytheoremintro}, we obtain a similar result for $\mathbb{Q}$-Calabi--Yau $3$-folds. 
Here, a {\it $\mathbb{Q}$-Calabi--Yau $3$-fold} is a normal projective $3$-fold with only terminal singularities whose canonical divisor is a torsion class. 
Let $r$ be the Gorenstein index of $X$, that is, the minimal positive integer such that $\mathcal{O}_X(rK_X) \simeq \mathcal{O}_X$. 
The isomorphism $\mathcal{O}_X(rK_X) \simeq \mathcal{O}_X$ determines the global index one cover 
$\pi \colon Y:= \Spec \oplus_{j=0}^{r-1} \mathcal{O}_X(j K_X) \rightarrow X$. 

As a consequence of Theorem \ref{coboundarytheoremintro} and the proof of \cite[Main Theorem 1]{mina}, we obtain the following. 

\begin{thm}\label{qcy3thm}
	Let $X$ be a $\mathbb{Q}$-Calabi--Yau $3$-fold. Assume that the global index one cover $Y \rightarrow X$ is $\mathbb{Q}$-factorial. 
	
	Then a $\mathbb{Q}$-Calabi--Yau $3$-fold $X$ can be deformed to one with only quotient singularities and 
	 $A_{1,2}/4$-singularities. 
	\end{thm}

\begin{rem}
	Namikawa  studied another invariant for terminal singularities and 
	$\mathbb{Q}$-smoothability of $\mathbb{Q}$-Calabi--Yau $3$-folds in his unpublished note. 
	The invariant is $\mu(X,x)$ defined in \cite[Section 2]{NamSt}. 
	It seems that this invariant also vanishes for  a $A_{1,2}/4$-singularity $(X,x)$. 
	So we do not know the $\mathbb{Q}$-smoothability of a $\mathbb{Q}$-Calabi-Yau $3$-fold with 
	$A_{1,2}/4$-singularities. 
	\end{rem}

\section{Calculation of coboundary maps}\label{coboundarysection}
First, we introduce the coboundary map of local cohomology which is used in \cite[3.2]{Sano} 
to find a $\mathbb{Q}$-smoothing of a $\mathbb{Q}$-Fano $3$-fold.  
(See also \cite[Section 1]{NamSt}, \cite[Section 4]{mina}.)

Let $(U,p)$ be a germ of a $3$-fold terminal singularity. 
Let $\pi_U \colon (V,q) \rightarrow (U,p)$ be the index one cover. By the classification (\cite{mori}, \cite{YPG}), we see that $(V,q)$ is a hypersurface singularity and 
$\pi_U$ is \'{e}tale outside $p$. Moreover, we have 
\[
(V,q) \simeq ((f=0),0) \subset (\mathbb{C}^4, 0)
\]
 for some $f \in \mathbb{C}[x,y,z,u]$, where 
$x,y,z,u$ are coordinate functions on $\mathbb{C}^4$ and $f$ satisfies $\sigma \cdot f = \zeta_U f$ for the generator $\sigma \in G:= \Gal (V/U) \simeq \mathbb{Z}_r$ 
and  $\zeta_U =\pm 1$.

We define the ordinariness of a terminal singularity as follows.
\begin{defn}\label{ordinarydefn}
	Let $(U,p)$ be a germ of a $3$-fold terminal singularity. 
	The germ $(U,p)$ is called {\it ordinary} (resp. {\it non-ordinary}) if $\zeta_U=1$ (resp. $\zeta_U= -1$). 
	\end{defn}

\begin{rem}
Let $(U,p)$ be a germ of a non-ordinary terminal singularity. By the classification (\cite{mori}, \cite{YPG}), we have 
\begin{equation}\label{nonordinarydescription}
(U,p) \simeq ((x^2 +y^2 +g(z,u)=0),0)/ \mathbb{Z}_4 \subset (\mathbb{C}^4/ \mathbb{Z}_4,0),  
\end{equation}
where $g(z,u) \in \mathfrak{m}_{\mathbb{C}^4,0}^2$ is some $\mathbb{Z}_4$-semi-invariant polynomial in $z, u$ and $\sigma \in \mathbb{Z}_4$ acts on $\mathbb{C}^4$ by $\sigma \cdot (x,y,z,u) \mapsto (\sqrt{-1}x, -\sqrt{-1} y, -z, \sqrt{-1}u)$. 
\end{rem}

Let $(U,p)$ be a germ of a $3$-fold terminal singularity and $V$ its index one cover with the $\mathbb{Z}_r$-action as above. 
Let $\nu \colon \tilde{V} \rightarrow V$ be a $\mathbb{Z}_r$-equivariant  resolution such that its exceptional divisor
 $F \subset \tilde{V}$ has SNC support and $\tilde{V} \setminus F \simeq V \setminus \{q \}$.  
Let $V':= V \setminus \{q \}$ and 
\[
\tau_V \colon H^1(V', \Omega^2_{V'}(-K_{V'})) \rightarrow H^2_{F}(\tilde{V}, \Omega^2_{\tilde{V}}(\log F)(-F-\nu^*K_V))
\] the coboundary map of the local cohomology. 
Note that the sheaf $\mathcal{O}_{V}(-K_V)$ and $\mathcal{O}_V$ are isomorphic as sheaves, 
but not isomorphic as $\mathbb{Z}_r$-equivariant sheaves.  
Let $\tilde{\pi} \colon \tilde{V} \rightarrow \tilde{U}:= \tilde{V}/\mathbb{Z}_r$ be the finite morphism induced by $\pi$ and $E \subset \tilde{U}$  
the exceptional locus of the birational morphism $\mu \colon \tilde{U} \rightarrow U$ induced by $\nu$. 
Let $U':= U \setminus \{ p \}$ and  $\mathcal{F}_U^{(0)}$ the $\mathbb{Z}_r$-invariant part of $\tilde{\pi}_* \Omega^2_{\tilde{V}}(\log F)(-F -\nu^* K_V)$. 
Then we have the coboundary map 
\begin{equation}\label{phiUdescription}
\phi_U \colon H^1(U', \Omega^2_{U'}(-K_{U'})) \rightarrow H^2_{E}(\tilde{U}, \mathcal{F}_U^{(0)}) 
\end{equation}
which is the $\mathbb{Z}_r$-invariant part of $\tau_V$. We shall study these coboundary maps $\tau_V$ and $\phi_U$ in this section. 

For an ordinary terminal singularity, we can calculate the map $\phi_U$ as follows. 

\begin{thm}\label{ordinarycoboundary}(cf.\ \cite[Lemma 3.4]{Sano}) 
	Let $(U,p)$ be a germ of a $3$-fold ordinary terminal singularity which is not a quotient singularity.  
	Then we have $\phi_U \neq 0$. 
	\end{thm}

In the following, we prepare ingredients for calculating $\phi_U$ for a germ $(U,p)$ of a non-ordinary terminal singularities. 

We have $H^1_{F}( \tilde{V}, \Omega^2_{\tilde{V}}(\log F)(-F)) =0$ by the proof of \cite[Theorem 4]{StDB}. 
We also have $H^2(\tilde{V}, \Omega^2_{\tilde{V}}(\log F)(-F))=0$ by the Guill\'{e}n--Navarro Aznar--Puerta--Steenbrink vanishing theorem. 
Thus we have an exact sequence 
\begin{multline}\label{rational-exact} 
0 \rightarrow H^1(\tilde{V}, \Omega^2_{\tilde{V}}(\log F) (-F -\nu^*K_V)) \rightarrow H^1(V', \Omega^2_{V'}(-K_{V'})) \\ 
\stackrel{\tau_V}{\rightarrow}  
H^2_{F}(\tilde{V}, \Omega^2_{\tilde{V}}(\log F)(-F - \nu^* K_V)) \rightarrow 0
\end{multline}

The following inequality proved in \cite{NamSt} is useful for the calculation of the coboundary maps.  

\begin{prop} We have 
\begin{equation}\label{ineqkerim}
\dim \Ker \tau_V \le \dim \Image \tau_V. 
\end{equation}
\end{prop}

\begin{proof}
This is proved in Remark after \cite[Theorem (1.1)]{NamSt}. 
Let us recall the proof for the convenience of the reader. 

By the exact sequence (\ref{rational-exact}), it is enough to show that 
\[
h^1(\tilde{V}, \Omega^2_{\tilde{V}}(\log F) (-F)) \le h^2_F(\tilde{V}, \Omega^2_{\tilde{V}}(\log F)(-F)). 
\]
We have a surjection 
\[
H^2_F(\tilde{V}, \Omega^2_{\tilde{V}}(\log F)(-F)) \rightarrow H^2_F(\tilde{V}, \Omega^2_{\tilde{V}}(\log F))
\]
since we have $H^2_F(\tilde{V}, \Omega^2_{\tilde{V}}(\log F)\otimes \mathcal{O}_F) = \Gr_F^2 H^5_{\{q\}}(V, \mathbb{C}) =0$. By the local duality, we have 
\[
H^2_F(\tilde{V}, \Omega^2_{\tilde{V}}(\log F))^* \simeq H^1(\tilde{V}, \Omega^1_{\tilde{V}}(\log F)(-F)). 
\]
Moreover we see that the differential homomorphism
\[
d \colon H^1(\tilde{V}, \Omega^1_{\tilde{V}}(\log F)(-F)) \rightarrow H^1(\tilde{V}, \Omega^2_{\tilde{V}}(\log F)(-F))
\]
is surjective by studying the spectral sequence 
\[
H^q(\tilde{V}, \Omega^p_{\tilde{V}}(\log F)(-F)) \Rightarrow \mathbb{H}^{p+q}(\tilde{V}, \Omega^{\bullet}_{\tilde{V}}(\log F)(-F)) =0   
\]
as in the proof of \cite[Theorem (1.1)]{NamSt}. Thus we obtain relations 
\begin{multline}
h^2_F(\tilde{V}, \Omega^2_{\tilde{V}}(\log F)(-F)) \ge h^2_F(\tilde{V}, \Omega^2_{\tilde{V}}(\log F)) = 
h^1(\tilde{V}, \Omega^1_{\tilde{V}}(\log F) (-F)) \\
 \ge h^1(\tilde{V}, \Omega^2_{\tilde{V}}(\log F) (-F)) 
\end{multline}
 and this implies (\ref{ineqkerim}). 
\end{proof}

Let $T^1_{(V,q)}$, $T^1_{(U,p)}$ be the sets of first order deformations of the germs $(V,q)$ and $(U,p)$ respectively. 
	Recall that we have an isomorphism $T^1_{(V,q)} \simeq \mathcal{O}_{V,q}/ J_{V,q}$ of $\mathcal{O}_{V,q}$-modules 
	for the Jacobian ideal $J_{V,q} \subset \mathcal{O}_{V,q}$. 
Hence we have a surjective $\mathcal{O}_{V,q}$-module homomorphism $\varepsilon \colon \mathcal{O}_{V,q} \rightarrow T^1_{(V,q)}$ 
which sends $h \in \mathcal{O}_{V,q}$ to the corresponding deformation $\varepsilon_h \in T^1_{(V,q)}$. 
Also we have a commutative diagram 
\[
\xymatrix{
T^1_{(U,p)} \ar[r]^{\simeq \ \ \ \ \ \ \ } \ar@{^{(}->}[d] & H^1(U', \Omega^2_{U'}(-K_{U'})) \ar@{^{(}->}[d] \\ 
T^1_{(V,q)} \ar[r]^{\simeq \ \ \ \ \ \ \ } & H^1(V', \Omega^2_{V'}(-K_{V'})),  
}
\]
where the horizontal isomorphisms are restrictions by open immersions and 
the upper terms inject into the lower terms as the $\mathbb{Z}_r$-invariant parts. 
Note that we have the horizontal isomorphisms since $\{p \} \hookrightarrow U$ and $\{q \} \hookrightarrow V$ 
have codimensions $3$, and the spaces $U$ and $V$ are Cohen-Macaulay. 
Thus we identify $T^1_{(V,q)}, T^1_{(U,p)}$ and $H^1(V', \Omega^2_{V'}(-K_{V'})), H^1(U', \Omega^2_{U'}(-K_{U'}))$ 
respectively via these isomorphisms.

We use the following notion of right equivalence (\cite[Definition 2.9]{GLSintro}).

\begin{defn}
	Let $\mathbb{C}\{x_1,\ldots, x_n \}$ be the convergent power series ring of $n$ variables. 
	Let $f,g \in \mathbb{C} \{x_1,\ldots, x_n \}$. 
	
	We say that $f$ is {\it right equivalent} to $g$ if there exists an automorphism $\varphi$ of 
	$\mathbb{C} \{x_1, \ldots , x_n \}$ such that $\varphi(f) = g$.  
	We write this as $f \overset{r}{\sim} g$. 
	\end{defn} 

By using these ingredients, we calculate the coboundary map for a non-ordinary singularity. 
The following theorem and Theorem \ref{ordinarycoboundary} imply Theorem \ref{coboundarytheoremintro}.  

\begin{thm}\label{coboundarytheorem}
Let $(U,p)$ be a germ of a non-ordinary $3$-fold terminal singularity which is not a quotient singularity. 
\begin{enumerate} 
\item[(i)] Assume that the index one cover $(V,q) \not\simeq ((x^2+y^2+z^3+u^2 =0),0)$.  
Then we have $\phi_U \neq 0$.  
\item[(ii)] Assume that $(V,q) \simeq  ((x^2+y^2+z^3+u^2 =0),0)$. Then $\phi_U =0$. 
\end{enumerate}
\end{thm}

\begin{proof}

(i) Suppose that $\phi_U =0$. We show the claim by contradiction. 
We can write $g(z,u) = \sum a_{i,j} z^i u^j \in \mathbb{C}[z,u]$ for some $a_{i,j} \in \mathbb{C}$ for $i,j \ge 0$.  
Since the generator $\sigma \in \mathbb{Z}_4$ acts on $g$ by $\sigma \cdot g = -g$ and on $z^i u^j$ by $\sigma \cdot z^i u^j = \sqrt{-1}^{2i +j} z^i u^j$, 
we see that $a_{i,j} \neq 0$ only if 
\begin{equation}\label{mod4cond}
2i+j \equiv 2 \mod 4.
\end{equation} 

Let $J_g:= (\frac{\partial g}{\partial z}, \frac{\partial g}{\partial u}) \subset \mathbb{C}[z,u]$ be 
the Jacobian ideal of the polynomial $g$. 
Note that we have $T^1_{(V,q)} \simeq \mathbb{C}[z,u]/(g, J_g)$ since $\varepsilon_x = \varepsilon_y =0 \in T^1_{(V,q)}$. 

({\bf Case 1}) Assume that $a_{0,2} \neq 0$. We can write 
\[
g(z,u) = u^2(1+ h_1(z,u)) + h_2(z)
\] 
for some  polynomials $h_1(z,u) \in (z,u) \subset \mathbb{C}[z,u]$ and $h_2(z) \in (z) \subset \mathbb{C}[z]$.  
Thus $g(z,u) \in \mathcal{O}_{\mathbb{C}^2,0}$ is right equivalent to $u^2 + h_2(z)$. 
We see that $h_2(z) \in \mathcal{O}_{\mathbb{C},0}$ is right equivalent to $z^{2i_0+1}$ for some positive integer $i_0$ since
$(g=0)$ has an isolated singularity and by the condition (\ref{mod4cond}).  
Thus we have \[
(V,q) \simeq ((x^2 +y^2 +z^{2i_0+1} + u^2=0),0).
\] 

If $i_0 =1$, it contradicts the assumption $(V,q) \not\simeq ((x^2+y^2+z^3+u^2 =0),0)$. 
Hence we have  $i_0 \ge 2$. 
By calculating the partial derivatives of $x^2+y^2 +z^{2 i_0 +1} +u^2$, we see that $\varepsilon_1, \varepsilon_z, \varepsilon_{z^2} \in T^1_{(V,q)}$ are linearly independent and 
\[
\dim T^1_{(V,q)} \ge 3.
\] 
On the other hand, we see that $\tau_V(\varepsilon_z) =0$ since we assumed $\phi_U =0$ and $\varepsilon_z \in T^1_{(U,p)}$. 
By this and the fact that $\tau_V$ is an $\mathcal{O}_{V,q}$-module homomorphism, we obtain a surjection 
$\mathbb{C}[z,u]/(z,u) \rightarrow \Image \tau_V$ since $ \varepsilon_u =0$.  
By this surjection and $\mathbb{C}[z,u]/(z,u) \simeq \mathbb{C}$,  we obtain $\dim \Image \tau_V \le 1$. 
By this and the inequality (\ref{ineqkerim}), we obtain an inequality 
\[
\dim T^1_{(V,q)} = \dim \Image \tau_V + \dim \Ker \tau_V \le 1+1 =2   
\] 
and it is a contradiction. 

({\bf Case 2}) Assume that $a_{0,2} =0$. Then we see that $a_{i,j} \neq 0$ only if $2i +j \ge 6$ by (\ref{mod4cond}). 
Note that a monomial $z^i u^j$ with $2i+j \ge 6$ is 
some multiple of either $z^3, z^2u^2, zu^4$ or $u^6$. 
By computing partial derivatives of these monomials, we see that  $(g,J_g) \subset (z^2, z u^2, u^4)$.
Thus we see that $\varepsilon_1, \varepsilon_z, \varepsilon_{zu}, \varepsilon_u, \varepsilon_{u^2}, \varepsilon_{u^3} \in T^1_{(V,q)}$ are linearly independent and 
we obtain 
\begin{equation}\label{toomuchineq}
\dim T^1_{(V,q)} \ge 6. 
\end{equation}

On the other hand, by the assumption $\phi_U =0$, we have $\tau_V(\varepsilon_z) =0, \tau_V( \varepsilon_{u^2}) =0$ 
since $\varepsilon_z, \varepsilon_{u^2} \in T^1_{(U,p)}$.  
Thus we have a relation $(z,u^2) \subset \Ker \tau_V \circ \varepsilon \subset \mathcal{O}_{V,q}$ and 
obtain a surjection $\mathbb{C}[z,u]/(z,u^2) \rightarrow \Image \tau_V$. 
This implies an inequality 
$\dim \Image \tau_V \le \dim \mathbb{C}[z,u]/(z,u^2) =2$. 
By this inequality and the inequality (\ref{ineqkerim}), we have an inequality 
\[
\dim T^1_{(V,q)} = \dim \Ker \tau_V + \dim \Image \tau_V \le 2+2 =4. 
\]
This contradicts (\ref{toomuchineq}). 

Hence we obtain $\phi_U \neq 0$ and  finish the proof of (i). 

(ii) For non-negative integers $i,j$, we set 
\[
b^{i,j}:= \dim H^j(\tilde{V}, \Omega^i_{\tilde{V}}(\log F)(-F)), 
\]  
\[
l^{i,j}:= \dim H^j(F, \Omega^i_{\tilde{V}}(\log F) \otimes \mathcal{O}_F).  
\] 
Let $s_k(V,q)$ for $k=0,1,2,3$ be the Hodge number of the Milnor fiber of $(V,q)$ as in \cite[Section 4]{StDB}. 
By \cite[Theorem 6]{StDB}, we have $s_0=0, s_1= b^{1,1}, s_2 =b^{1,1} +l^{1,1}$ and $s_3 = l^{0,2}$. We see that $l^{0,2} =0$ by \cite[Lemma 2]{StDB}. 
  Since the sum $\sum_{k=0}^3 s_k(V,q)$ is the Milnor number of $(V,q)$, we obtain  $2b^{1,1} +l^{1,1} = 2$. 
Since $b^{1,1} \neq 0$ by \cite[Theorem 2.2]{NamSt}, we obtain 
\begin{equation}\label{b11l11} 
	b^{1,1} =1, \ \ l^{1,1}=0.
	\end{equation} 

There exists an exact sequence 
\begin{multline}
H^0(F, \Omega^1_{\tilde{V}}(\log F) \otimes \mathcal{O}_F)  \rightarrow H^1(\tilde{V}, \Omega^1_{\tilde{V}}(\log F)(-F)) \rightarrow H^1(\tilde{V}, \Omega^1_{\tilde{V}}(\log F))   \\ \rightarrow 
H^1(F, \Omega^1_{\tilde{V}}(\log F) \otimes \mathcal{O}_F). 
\end{multline}
Since $l^{1,0} =0$ by \cite[Lemma 1]{StDB}, the both outer terms are zero and the homomorphism in the middle is an isomorphism. 
By this and (\ref{b11l11}), 
we have \begin{equation}\label{h2F1dim}
\mathbb{C} \simeq H^1(\tilde{V}, \Omega^1_{\tilde{V}}(\log F)) \simeq H^2_F(\tilde{V},  \Omega^2_{\tilde{V}}(\log F)(-F))^*.
\end{equation} 

Suppose that $\tau_V(\varepsilon_z) \neq 0$. Then $\varepsilon_z \not\in \Ker \tau_V$. 
This implies that $\Ker \tau_V =0$ since $T^1_{(V,q)} \simeq \mathbb{C}[z]/(z^2)$ as $\mathbb{C}[z]$-modules. 
Thus $\mathbb{C}^2 \simeq \Image \tau_V \simeq H^2_F(\tilde{V},  \Omega^2_{\tilde{V}}(\log F)(-F))$.    
This contradicts (\ref{h2F1dim}).  

Thus we obtain $\tau_V(\varepsilon_z) =0$. 
Since $T^1_{(U,p)} \simeq \mathbb{C}$ is generated by $\varepsilon_z$, we see that $\phi_U=0$. 
Thus we finish the proof of (ii). 
\end{proof}

\vspace{5mm}

Now we prepare another coboundary map to study $\mathbb{Q}$-smoothability of a $\mathbb{Q}$-Calabi--Yau $3$-fold. 

Let $(U,p)$ be a germ of a $3$-fold terminal singularity and $V, \tilde{V}, F, \tilde{U}$ as before. 
We have the coboundary map 
\[
\bar{\tau}_V  \colon H^1(V', \Omega^2_{V'}(-K_{V'})) \rightarrow H^2_F(\tilde{V}, \Omega^2_{\tilde{V}}(-\nu^* K_V))
\]
and this fits in the commutative diagram 
\begin{equation}\label{tau_Vfactor}
\xymatrix{
H^1(V', \Omega^2_{V'}(-K_{V'})) \ar[r]^{\bar{\tau}_V} \ar[d]^{\tau_V} & H^2_F(\tilde{V}, \Omega^2_{\tilde{V}}(-\nu^* K_V)) \\
H^2_F(\tilde{V}, \Omega^2_{\tilde{V}}(\log F)(-F - \nu^* K_V)), \ar@{^{(}->}[ru]_{\tau'_V} & 
}
\end{equation}
where the injectivity of $\tau'_V$ is proved in the proof of \cite[Theorem 1.1]{NamSt}. 

Let $\bar{\mathcal{F}}^{(0)}_U:= (\tilde{\pi}_* \Omega^2_{\tilde{V}}(-\nu^* K_V))^{\mathbb{Z}_r}$ 
be the $\mathbb{Z}_r$-invariant part. 
Let 
\[
\bar{\phi}_U \colon H^1(U', \Omega^2_{U'}(-K_{U'})) \rightarrow H^2_E(\tilde{U}, \bar{\mathcal{F}}_{U}^{(0)})
\]
be the coboundary map. It is the $\mathbb{Z}_r$-invariant part of $\bar{\tau}_V$. 
As the $\mathbb{Z}_r$-invariant part of the diagram (\ref{tau_Vfactor}), we obtain the following diagram;
\begin{equation*}
\xymatrix{
H^1(U', \Omega^2_{U'}(-K_{U'})) \ar[r]^{\bar{\phi}_U} \ar[d]^{\phi_U} & H^2_E(\tilde{U}, \bar{\mathcal{F}}_{U}^{(0)}) \\
H^2_E(\tilde{U},\mathcal{F}_{U}^{(0)}). \ar@{^{(}->}[ru]_{\phi'_U} & 
}
\end{equation*}
By these arguments, we obtain the following result as a corollary of 
Theorem \ref{ordinarycoboundary} and Theorem \ref{coboundarytheorem}. 

\begin{cor}\label{barphicorollary}
	Let $(U,p)$ be a germ of a $3$-fold terminal singularity which is not a quotient singularity.

	Then $\bar{\phi}_U =0$ if and only if the germ $(U,p)$ is an $A_{1,2}/4$-singularity. 
	\end{cor}

We use the blow-down morphism of deformations by a resolution $\tilde{V} \rightarrow V$ 
to find a $\mathbb{Q}$-smoothing. 
It is already used in several papers on deformations of singular $3$-folds. (cf.\ \cite{NamSt}, \cite{NamFano}, \cite{Sano})

Let 
\[
\nu_* \colon H^1(\tilde{V}, \Omega^2_{\tilde{V}}(-K_{\tilde{V}})) \rightarrow H^1(V', \Omega^2_{V'}(-K_{V'}))  
\]
be the restriction homomorphism by the open immersion $V' \hookrightarrow \tilde{V}$. 
We use this notation since there is a commutative diagram 
\[
\xymatrix{
H^1(\tilde{V}, \Omega^2_{\tilde{V}}(-K_{\tilde{V}})) \ar[r]^{\nu_*} \ar[d]^{\simeq} 
&  H^1(V', \Omega^2_{V'}(-K_{V'}))  \ar[d]^{\simeq} \\ 
T^1_{\tilde{V}} \ar[r] & T^1_V,  
}
\]
where the lower horizontal homomorphism is the blow-down homomorphism of deformations (\cite{Wahl}). 
We can prove the relation 
\begin{equation}\label{blowdownrel}
	\Image \nu_* \subset \Ker \tau_V = \Ker \bar{\tau}_V    
	\end{equation}
by the same argument as in \cite[Claim 3.7]{Sano}.

\section{Application to $\mathbb{Q}$-smoothing problems}
In \cite[Theorem 3.2]{Sano}, we proved the following. 

\begin{thm}\label{thmphi0}
Let $X$ be a $\mathbb{Q}$-Fano $3$-fold. 

Then there exists a deformation $\mathcal{X} \rightarrow \Delta^1$ of $X$ 
over a unit disc $\Delta^1$ such that the general fiber $\mathcal{X}_t$ for $t \in \Delta^1 \setminus \{0 \}$ satisfies the following; 
For each singular point $p \in \mathcal{X}_t$ and its Stein neighborhood $U_p$, the coboundary map $\phi_{U_p}$ vanishes. 
\end{thm}

As an application of this result and Theorem \ref{coboundarytheorem}, we obtain a proof of Theorem \ref{qsmqfanothmintro} as follows. 

\begin{proof}[Proof of Theorem \ref{qsmqfanothmintro}]
	By Theorem \ref{thmphi0}, we can deform a $\mathbb{Q}$-Fano $3$-fold $X$ to one with only singularities $p_1, \ldots, p_l$ 
	such that $\phi_{U_i} =0$, where $U_i$ is a Stein neighborhood of $p_i$ for $i=1, \ldots, l$. By Theorem \ref{coboundarytheoremintro}, 
	such a terminal singularity is either a quotient singularity or an $A_{1,2}/4$-singularity. Thus we finish the proof.   
	\end{proof}

\vspace{2mm}

\begin{eg}\label{eg:A12hypersurface}
	There exists an example of a $\mathbb{Q}$-Fano $3$-fold with an $A_{1,2}/4$-singularity. 
	This example has a $\mathbb{Q}$-smoothing. 
	
	Let $X:=X_{10}\subset \mathbb{P}(1,1,2,3,4)$ be a weighted hypersurface of degree $10$ defined by the polynomial 
	\[
	f_{X_{10}}:= w^2(x_1^2 + x_2^2) +w(y^3+z^2) +x_1^{10} +x_2^{10} +y^5 + z^3 x_1, 
	\]
	where $x_1, x_2, y, z, w$ are coodinates of weights $1,1,2,3,4$, respectively. 
	By perturbing the coefficients of the polynomial, we obtain that 
	\[
	\Sing X = \{[0:0:0:1:0], [0:0:0:0:1] \},  
	\]
	$p_z:= [0:0:0:1:0]$ is a $1/3(1,1,2)$-singularity and $p_w:= [0:0:0:0:1]$ is an $A_{1,2}/4$-singularity. 
Let 
\[
\mathcal{X}:= (f_{X_{10}}+ t \cdot yw^2=0) \subset \mathbb{P}(1,1,2,3,4) \times \mathbb{A}^1 \rightarrow \mathbb{A}^1
\]
be a deformation of $X$, where $t$ is a coordinate of $\mathbb{A}^1$. 
Then we see that $\mathcal{X}$ is a $\mathbb{Q}$-smoothing of $X$. 
The general fiber $\mathcal{X}_{t}$ has two $1/2(1,1,1)$-singularities, a $1/3(1,1,2)$-singularity and a $1/4(1,3,1)$-singularity.  
	\end{eg}

\begin{rem}\label{A12remark}
We give a comment on a $\mathbb{Q}$-Fano $3$-fold with $A_{1,2}/4$-singularities.

Let $X$ be a $\mathbb{Q}$-Fano $3$-fold. The local-to-global spectral sequence of $\Ext$ groups induces an exact sequence 
\[
\Ext^1(\Omega^1_X, \mathcal{O}_X) \rightarrow H^0(X, \underline{\Ext}^1(\Omega^1_X, \mathcal{O}_X) ) 
\rightarrow H^2(X, \Theta_X),  
\]
where $\underline{\Ext}^1$ is a sheaf of $\Ext$ groups. 
Recall that $\Ext^1(\Omega^1_X, \mathcal{O}_X)$ and $ H^0(X, \underline{\Ext}^1(\Omega^1_X, \mathcal{O}_X) )$ are 
the sets of first order deformations of $X$ and the singularities on $X$, respectively. 
Thus, if we have $H^2(X, \Theta_X) =0$, we see that $X$ is $\mathbb{Q}$-smoothable. 

However, this approach does not work in general. 
Namikawa constructed an example of a Fano $3$-fold $X$ with $A_{1,2}$-singularities such that $H^2(X, \Theta_X) \neq 0$  
 (\cite[Example 5]{NamFano}). Here an $A_{1,2}$-singularity 
 is a hypersurface singularity locally isomorphic to $(x^2+y^2+z^3+u^2=0) \subset \mathbb{C}^4$.
This $X$ has a smoothing. 
The author expects that there also exists a $\mathbb{Q}$-Fano $3$-fold $X$ with $A_{1,2}/4$-singularities such that 
$H^2(X, \Theta_X) \neq 0$. 

Thus we do not know $\mathbb{Q}$-smoothability of a $\mathbb{Q}$-Fano $3$-fold with $A_{1,2}/4$-singularities.
\end{rem}

\vspace{5mm}

As another application of Theorem \ref{coboundarytheorem}, we obtain a proof of Theorem \ref{qcy3thm} as follows. 
	
	\begin{proof}[Proof of Theorem \ref{qcy3thm}]
		The proof is a modification of the proof of \cite[Main Theorem 1]{mina}.
We sketch the proof for the convenience of the reader.

First we prepare notations to define the diagram (\ref{minadiagqcy3}).

Let $p_1,\ldots, p_l \in X$ be the non-quotient singularities and $U_1, \ldots, U_l$ their Stein neighborhoods. 
Let $\nu \colon \tilde{Y} \rightarrow Y$ be a $\mathbb{Z}_r$-equivariant resolution such that  
its exceptional divisor $F$ is a SNC divisor and $\tilde{Y} \setminus F \simeq Y \setminus \nu^{-1}(\{p_1, \ldots, p_l \})$. 
Let $\tilde{\pi} \colon \tilde{Y} \rightarrow \tilde{X}:= \tilde{Y}/ \mathbb{Z}_r$ be the quotient morphism 
and $\mu \colon \tilde{X} \rightarrow X$ the induced birational morphism with the exceptional divisor $E$. 

Let $V_i:= \pi^{-1}(U_i)$, $\tilde{V}_i:= \nu^{-1}(V_i)$, $F_i:= F \cap \tilde{V_i}$ and $\nu_i:= \nu|_{\tilde{V}_i} \colon \tilde{V}_i \rightarrow V_i$ be the restrictions. 
Let $\tilde{U}_i:= \mu^{-1}(U_i)$, $E_i := E \cap \tilde{U}_i$ and $\tilde{\pi}_i:= \tilde{\pi}|_{\tilde{V}_i} \colon \tilde{V}_i \rightarrow \tilde{U}_i$ 
the induced finite morphism. 
Let $\bar{\mathcal{F}}^{(0)}:= \left( \tilde{\pi}_* \Omega^2_{\tilde{Y}}(- \nu^* K_V) \right)^{\mathbb{Z}_r}$ be the $\mathbb{Z}_r$-invariant part 
and $\bar{\mathcal{F}}_i^{(0)}:= \bar{\mathcal{F}}^{(0)}|_{\tilde{U}_i}$ its restriction. 

Then we have the diagram 
\begin{equation}\label{minadiagqcy3}
	\xymatrix{
H^1(X', \Omega^2_{X'}(-K_{X'})) \ar[r]^{\oplus \psi_i} \ar[d]^{\oplus p_{U_i}} & 
\oplus_{i=1}^l H^2_{E_i}(\tilde{X}, \bar{\mathcal{F}}^{(0)}) \ar[r]^{\oplus B_i} \ar[d]_{\oplus \varphi_i}^{\simeq} & 
H^2(\tilde{X}, \bar{\mathcal{F}}^{(0)}) \\
\oplus_{i=1}^l H^1(U'_i, \Omega^2_{U'_i}(-K_{U'_i})) \ar[r]^{\bar{\phi}_i} & 
\oplus_{i=1}^l H^2_{E_i}(\tilde{U}_i, \bar{\mathcal{F}}_i^{(0)}), & 
}
\end{equation}
where $X':= X \setminus \{p_1, \ldots, p_l \}$ and $U_i' := U_i \cap X'$. 

Let $V'_i := \pi^{-1}(U'_i)$. 
Note that $B_i \circ \varphi_i^{-1} \circ \bar{\phi}_i$ is the $\mathbb{Z}_r$-invariant part of the composition 
\begin{multline}
H^1(V'_i, \Omega^2_{V'_i}(-K_{V'_i})) \rightarrow H^2_{F_i}(\tilde{V}_i, \Omega^2_{\tilde{V}_i}(-\nu_i^* K_{V_i})) 
\rightarrow H^2_{F_i}(\tilde{Y}, \Omega^2_{\tilde{Y}}(-\nu^* K_{Y})) \\
\rightarrow H^2(\tilde{Y}, \Omega^2_{\tilde{Y}}(-\nu^* K_{Y})). 
\end{multline}
We see that this is zero by \cite[Proposition 1.2]{NamSt} since we assumed that $Y$ is $\mathbb{Q}$-factorial. 
Thus we also see that $B_i \circ \varphi_i^{-1} \circ \bar{\phi}_i =0$. 

There exists an element $\eta_i \in H^1(U'_i, \Omega^2_{U'_i}(-K_{U'_i}))$ such that $\bar{\phi}_i(\eta_i) \neq 0$ by Theorem \ref{coboundarytheoremintro}. 
Since $B_i \circ \varphi_i^{-1} \circ \bar{\phi}_i(\eta_i) =0$, there exists 
$\eta \in H^1(X', \Omega^2_{X'}(-K_{X'}))$ such that $\psi_i(\eta)= \varphi_i^{-1}(\phi_i(\eta_i))$. 
By the relation (\ref{blowdownrel}) and $p_{U_i}(\eta) - \eta_i \in \Ker \bar{\phi}_i$, 
we see that $p_{U_i}(\eta) \not\in \Image (\nu_i)_*$, where we use the inclusion 
$H^1(U'_i, \Omega^2_{U'_i}(-K_{U'_i})) \subset H^1(V'_i, \Omega^2_{V'_i}(-K_{U'_i}))$. 
By arguing as in the proof of \cite[Theorem 3.5]{Sano}, we can deform singularity $p_i \in U_i$ as long as 
$\bar{\phi}_i \neq 0$. 
By Corollary \ref{barphicorollary}, we obtain a required deformation since the deformations of a $\mathbb{Q}$-Calabi--Yau 
$3$-fold are unobstructed (\cite[Theorem A]{Namtop}).

		\end{proof}

\section*{Acknowledgments} 
This paper is a part of the author's Ph.D thesis submitted to University of Warwick. 
The author would like to express deep gratitude to Prof. Miles Reid for his warm encouragement and 
valuable comments. 
He would like to thank Professor Yoshinori Namikawa for useful conversations. 
Part of this paper is written during the author's stay in Princeton university and the university of Tokyo. 
He would like to thank Professors J\'{a}nos Koll\'{a}r and Yujiro Kawamata for useful comments and nice hospitality. 
He thanks the referee for useful suggestions. 
He is partially supported by Warwick Postgraduate Research Scholarship.

\end{document}